\documentclass[12pt,leqno,fleqn]{amsart}  
\usepackage{tikz}
\usepackage{amsmath,amstext,amsthm,amssymb,amsxtra}
\usepackage[top=1.5in, bottom=1.5in, left=1.25in, right=1.25in]	{geometry}
\usepackage{txfonts} 
\usepackage[T1]{fontenc}
\usepackage{lmodern}

\numberwithin{equation}{section}

\usepackage{mathtools}
\mathtoolsset{showonlyrefs,showmanualtags} 

\usepackage[msc-links]{amsrefs}  

\usepackage{hyperref} 
\hypersetup{
	colorlinks=true,       
	linkcolor=blue,          
	citecolor=magenta,        
	filecolor=magenta,      
	urlcolor=cyan           
}

\newtheorem{theorem}{Theorem}[section]

\newtheorem{lemma}{Lemma}[section]

\newtheorem{definition}{Definition}[section]

\def\ZS{\ensuremath{\mathcal S}}
\def\ZM{\ensuremath{\mathfrak M}}

\def\ZB{\ensuremath{\mathfrak B}}
\def\ZA{\ensuremath{\mathcal A}}
\def\ZZ{\ensuremath{\mathbb Z}}
\def\ZI{\ensuremath{\mathbb I}}

\def\ZK{\ensuremath{\mathcal K}}
\def\ZR{\ensuremath{\mathbb R}}
\def\pr{\ensuremath{\mathrm {pr}}}

\def\ZG{{\mathcal G\,}}

\newcommand {\e }[1]{\eqref{#1}}
\newcommand {\lem }[1]{Lemma \ref{#1}}

\newcommand {\trm }[1]{Theorem \ref{#1}}

\begin{document}
\title[]{On a weak type estimate for sparse operators of strong type}

\author{Grigori A. Karagulyan}
\address{Faculty of Mathematics and Mechanics, Yerevan State
	University, Alex Manoogian, 1, 0025, Yerevan, Armenia} 
\email{g.karagulyan@ysu.am}

\author{Gevorg Mnatsakanyan}
\address{Faculty of Mathematics and Mechanics, Yerevan State
	University, Alex Manoogian, 1, 0025, Yerevan, Armenia} 
\email{mnatsakanyan\_g@yahoo.com}

\thanks{Research was  supported by a grant from Science Committee of Armenia 18T-1A081. Part of this research was carried out at the American Institute of Mathematics, during a workshop on `Sparse Domination of Singular Integrals', October 2017.} 
\date{}
\maketitle
\begin{abstract}
	We define sparse operators of strong type on abstract measure spaces with ball-bases. Weak and strong type inequalities for such operators are proved. 
\end{abstract}
\section{Introduction}
The sparse operators are very simple positive operators recently appeared in the study of weighted estimates of Calder\'{o}n-Zygmund and other related operators. It was proved that some well-known operators (Calder\'{o}n-Zygmund operators, martingale transforms, maximal function, Carleson operators, etc.) can be dominated by sparse operators, and this kind of dominations derive series of deep results for the mentioned operators \cite{CoRe, PlLe, Lac1, Kar1,  Ler1,Ler2}. In particular, Lerner's \cite{Ler1} norm domination of the Calder\'{o}n-Zygmund operators by sparse operators gave a simple alternative proof to the $A_2$-conjecture solved by Hyt\"{o}nen \cite{Hyt}. Lacey \cite {Lac1} established a pointwise sparse domination for the Calder\'{o}n-Zygmund operators with an optimal condition (Dini condition) on the modulus of continuity, getting a logarithmic gain to the result previously proved by Conde-Alonso and Rey \cite{CoRe}. The paper \cite{Lac1} also proves a pointwise sparse domination for the martingale transforms, providing a short approach to the $A_2$-theorem proved by Treil-Thiele-Volberg \cite{TTV}. 
For the Carleson operators the norms sparse domination was proved by Di Plinio and Lerner \cite{PlLe}, while the pointwise domination follows from a general result proved later in \cite{Kar1}. 
 
In this paper we consider sparse operators based on ball-bases in abstract measure spaces. The concept of ball-basis was introduced by the first author in \cite{Kar1}. Based on ball-basis the paper \cite {Kar1} defines a wide class of operators (including the above mentioned operators, in particular) that can be pointwisely dominated by sparse operators. Some estimates of sparse operators in abstract spaces were shown in \cite{Kar1}. In this paper we will define a stronger version of sparse operators. We will prove weak and strong type estimates for such operators. 

Recall the definition of the ball-basis from \cite{Kar1}.
\begin{definition} Let $(X,\ZM,\mu)$ be a measure space. A family of sets $\ZB\subset \ZM$ is said to be a ball-basis if it satisfies the following conditions: 
	\begin{enumerate}
		\item[B1)] $0<\mu(B)<\infty$ for any ball $B\in\ZB$.
		\item[B2)] For any points $x,y\in X$ there exists a ball $B\ni x,y$.
		\item[B3)] If $E\in \ZM$, then for any $\varepsilon>0$ there exists a finite or infinite sequence of balls $B_k$, $k=1,2,\ldots$, such that
		\begin{equation*}	
		\mu\left(E\bigtriangleup \bigcup_k B_k\right)<\varepsilon.
		\end{equation*} 
		\item[B4)] For any $B\in\ZB$ there is a ball $B^*\in\ZB $ (called {\rm hull} of $B$) satisfying the conditions
		\begin{align}
		&\bigcup_{A\in\ZB:\, \mu(A)\le 2\mu(B),\, A\cap B\neq\varnothing}A\subset B^*,\label{h12}\\
		&\qquad\qquad\mu(B^*)\le \ZK\mu(B),\label{h13}
		\end{align}
		where $\ZK$ is a positive constant.
	\end{enumerate}
\end{definition}
A ball-basis $\ZB$ is said to be doubling if there is a constant $\eta>1$ such that for any $A\in \ZB$, $A^{*}\neq X$, one can find a ball $B\in \ZB$ satisfying
\begin{equation}\label{h73}
A\subsetneq B,\quad \mu(B)\le\eta  \cdot \mu(A).
\end{equation}
It is shown in \cite {Kar1} that condition \e {h73} in the definition can be equivalently replaced by a stronger condition $\eta_1\le \mu(B)/\mu(A)\le \eta_2$, where $\eta_2>\eta_1>1$. It is well-known the non-standard futures of non-doubling bases in many problems of analysis.
 
One can easily check that the family of Euclidean balls in $\ZR^n$ forms a ball-basis and it is doubling. 
An example of non-doubling ball-basis can serve us the martingale-basis defined as follows: Let $(X,\ZM,\mu)$ be a measure space, and let $\{\ZB_n:\, n\in\ZZ\}$ be a collections of measurable sets such that 1) each $\ZB_n$ is a finite or countable partition of $X$, 2) for each $n$ and $A\in \ZB_n$  the set $A$ is the union of sets $A'\in \ZB_{n+1}$, 3) the collection $\ZB=\cup_{n\in\ZZ}\ZB_n$ generates the $\sigma$-algebra $\ZM$, 4) for any points $x,y\in X$ there is a set $A\in \ZB$ such that $x,y\in A$. One can easily check that $\ZB$ satisfies the ball-basis conditions B1)-B4). On the other hand it is not always doubling. Obviously, it is doubling if and only if $\mu(\pr(B))\le c\mu(B)$, $B\in \ZB$, where $\pr(B)$ (parent of $B$) denotes the minimal ball satisfying $B\subsetneq \pr(B)$.

Hence let $\ZB$ be a ball basis in a measure space $(X,\ZM,\mu)$. For $f\in L^r(X)$, $1\le r<\infty$, and a  ball $B\in \ZB$ we set 
\begin{equation*}
\langle f\rangle_{B,r}=\left(\frac{1}{\mu(B)}\int_{B}|f|^r\right)^{1/r},\quad \langle f\rangle^*_{B,r}=\sup_{A\in \ZB:A\supset B}\langle f\rangle_{A,r}.
\end{equation*}
A collection of balls $\ZS\subset \ZB$ is said to be sparse or $\gamma$-sparse if for any $B\in \ZS$ there is a set $E_B\subset B$ such that $\mu(E_B)\ge \gamma\mu(B)$ and the sets $\{E_B:\, B\in \ZS\}$ are pairwise disjoint, where  $0<\gamma<1$ is a constant.   
We associate with $\ZS$ the operators 
\begin{align*}
&\ZA_{\ZS,r}f(x)=\sum_{A\in \ZS}\left\langle f\right\rangle_{A,r}\cdot\ZI_A(x),\\
&\ZA_{\ZS,r}^*f(x)=\sum_{A\in \ZS}\langle f\rangle^*_{A,r}\cdot \ZI_{A}(x),
\end{align*}
called sparse and strong type sparse operators respectively. The weak-$L^1$ estimate of $\ZA_{\ZS,1}$ in $\ZR^n$ (case $r=1$) as well as the boundedness on $L^p$ ($1<p<\infty$)  were proved by Lerner \cite{Ler1}. The $L^p$-boundedness of $\ZA_{\ZS,r}$ for general ball-bases was shown by the first author \cite{Kar1}. 

We will say a constant is admissible if it depends only on $p$ and on the constants $\ZK$ and $\gamma$ from the above definitions, and the notation $a\lesssim b$ will stand for the inequality $a\le c \cdot b$, where $c>0$ is an admissible constant. The main result of the paper is the weak-$L^r$ estimate of $\ZA_{\ZS,r}^*$ generated by general ball-bases.  That is
\begin{theorem}\label{T1}
	A sparse operator of strong type $\ZA_{\ZS,r}^*$, $1\le r<\infty$, corresponding to a general ball-basis is bounded operator on $L^p$ for $r<p<\infty$, and satisfies the weak-$L^r$ estimate. That is
	\begin{align}
	&\left\|\ZA_{\ZS,r}^*(f)\right\|_{p}\lesssim \|f\|_p,\quad r<p<\infty, \\
	& \mu\left\{\ZA_{\ZS,r}^*(f)>\lambda  \right\}\lesssim \frac{\|f\|_r^r}{\lambda^r},\quad \lambda>0.\label{a20}
	\end{align}
\end{theorem}
The proof of $L^p$-boundedness of $\ZA_{\ZS,r}^*$ is simple and uses the duality argument likewise \cite{Ler1}. Lerner's \cite {Ler1} proof of weak-$L^1$ estimate in $\ZR^n$ applies the standard Calder\'on-Zygmund decomposition argument. The Calder\'on-Zygmund decomposition may fail if the ball-basis is not doubling, so for the weak-$L^r$ estimate in the case of general ball-basis we apply the function flattening technique displayed in \lem {L3}. That is, we reconstruct the function $f\in L^r$ round the big values to get a $\lambda$-bounded function $g\in L^{2r}$, having ball averages of $f$ dominated by those of $g$. As a result we will have $\|\ZA_{\ZS,r}^*f\|_{r,\infty}\lesssim \|\ZA_{\ZS,r}^*g\|_{2r,\infty}$, reducing the weak-$L^r$ estimate of $\ZA_{\ZS,r}^*$ to weak-$L^{2r}$.

\section{Auxiliary lemmas}
Recall some definitions and propositions from \cite{Kar1}. We say that a set $E\subset X$ is bounded if $E\subset B$ for a ball $B\in \ZB$.
\begin{lemma}[\cite{Kar1}]\label{OL1}
	Let $(X,\ZM,\mu)$ be a measure space with a ball-basis $\ZB$. If $E\subset X$ is bounded and $\ZG $ is a family of balls with
	\begin{equation*}
	E\subset \bigcup_{G\in \ZG}G,
	\end{equation*}
	then there exists a finite or infinite sequence of pairwise disjoint balls $G_k\in \ZG$ such that
	\begin{equation}\label{c1}
	E \subset \bigcup_k G_k^*.
	\end{equation}
\end{lemma}
\begin{definition}
	For a set $E\in \ZM $ a point $x\in E$ is said to be density point if for any $\varepsilon>0$ there exists a ball $B\ni x$ such that
	\begin{equation*}
	\mu(B\cap E)>(1-\varepsilon )\mu(B).
	\end{equation*} 
	We say a measure space $(X,\ZM,\mu)$ satisfies the density property if almost all points of any measurable set are density points. 
\end{definition}

\begin{lemma}[\cite{Kar1}]\label{OL2}
	Any ball-basis satisfies the density property.
\end{lemma}
Denote the $L^r$ maximal function associated to the ball-basis $\ZB$ by
\begin{equation*}
M_rf(x)=\sup_{B\in \ZB:\,x\in B }\langle f\rangle_{B,r}
\end{equation*}
\begin{lemma}[\cite{Kar1}]\label{OL3}
	If $1\le r<p\le \infty$, then the maximal function $M_r$ satisfies the strong $L^p$ and weak-$L^r$ inequalities.
\end{lemma}
 \begin{definition}
 	We say $B\in\ZB$ is a $\lambda$-ball for a function $f\in L^r(X)$ if 
 	\begin{equation*}
 	\langle f\rangle_{B,r}> \lambda.
 	\end{equation*}
 	If, in addition, there is no $\lambda$-ball $A\supset B$ satisfying $\mu(A)\ge 2\mu(B)$,
 	then $B$ is said to be maximal $\lambda$-ball for $f$.
 \end{definition}

 \begin{lemma}\label{L2}
 	Let the function $f\in L^r(X)$ have bounded support and $\lambda>0$. There exist pairwise disjoint maximal $\lambda$-balls $\{B_k\}$ such that
 	\begin{equation}\label{a1}
 	G_\lambda= \{x\in X:\,M_rf(x)>\lambda\} \subset  \bigcup_kB_k^*.
 	\end{equation}
 \end{lemma}
\begin{proof}  
	Since $f$ has bounded support, one can easily check that the set $G_\lambda$
is also bounded. Besides, any $\lambda$-ball is in some maximal $\lambda$-ball. Thus we conclude that $G_\lambda=\bigcup_\alpha B_\alpha$, where each $B_\alpha$ is a maximal $\lambda$-ball. Applying the above covering lemma, we find a sequence of pairwise disjoint balls $B_k$
such that 
\begin{equation*}
G_\lambda \subset \bigcup_kB_k^*
\end{equation*}
and so we will have \e {a1}. 
\end{proof}
Let $B\subset (a,b)$ be a Lebesgue measurable set. For a given positive real $\kappa\le |B|$ denote  
\begin{equation*}
a(\kappa, B)=\inf \{a':\, |(a,a')\cap B)|\ge \kappa\},\quad L(\kappa,B)=(a,a(\kappa,G))\cap B.
\end{equation*}
Observe that $L(\kappa,B)$ determines the "leftmost" set of measure $\kappa$ in $B$ and $a(\kappa, B)$ does not depend on the choice of $a$.  
\begin{lemma}\label{L4}
Let $A\subset B\subset (a,b)$ be Lebesgue measurable sets on the real line and $0<\kappa \le |A|$. Then we have
\begin{equation*}
|L(\kappa,B)\triangle L(\kappa, A)|\le 2|B\setminus A|.
\end{equation*}
\end{lemma}
\begin{proof}
 Obviously, $a\le a(\kappa,B)\le a(\kappa,A)\le b$. Since $|L(\kappa,B)|=|L(\kappa,B)|$, the sets
	\begin{align*}
	&L(\kappa,B)\setminus L(\kappa, A)=\big((a,a(\kappa,B))\cap (B\setminus A) \big),\\
	&L(\kappa,A)\setminus L(\kappa, B)=\big((a(\kappa,B),a(\kappa,A))\cap A\big).
	\end{align*}
have the same measure. So we get
	\begin{equation*}
	|L(\kappa,B)\triangle L(\kappa, A)|=2\left|\big((a,a(\kappa,B))\cap (B\setminus A) \big)\right|\le 2|B\setminus A|.
	\end{equation*}
\end{proof}
\begin{lemma}\label{L5}
	Let $(X,\ZM,\mu)$ be a non-atomic measure space and $G_k$ be a finite or infinite sequence of measurable sets in $X$. If a sequence of numbers $\xi_k\ge 0$ satisfies $\sum_k\xi_k<\infty$ and the condition 
	\begin{equation}\label{a36}
	 \bigcup_{j:\, \mu(G_j)\le \mu(G_k),\,G_j\cap G_k\neq\varnothing}\xi_j\le \mu(G_k),\quad k=1,2,\ldots,
	\end{equation} 
	then there exist pairwise disjoint measurable sets $\tilde G_k\subset G_k$ such that 
	\begin{equation}\label{a37}
	\mu(\tilde G_k)=\xi_k, k=1,2,\ldots .
	\end{equation}
	
\end{lemma}	
\begin{proof}
	Without loss of generality we can suppose that $\mu(G_k)$ is decreasing. Since the measure space is non-atomic, we can also suppose that $G_k$ are Lebesgue measurable sets in $\ZR$. First we assume that the sequence $G_k$, $k=1,2,\ldots,n$, is finite. We apply backward induction. The existence of $\tilde G_n\subset G_n$ satisfying $\mu(\tilde G_n)=\xi_n$ follows from \e {a36}, since the latter implies $\xi_n\le \mu(G_n)$ and we have the measure is non-atomic. We will define $\tilde G_n$ to be the leftmost set in $G_n$ that is $\tilde G_n=L(\xi_n,G_n)$. Suppose by induction we have defined pairwise disjoint sets $\tilde G_k\subset G_k$ satisfying \e {a37} for $l\le k\le n$. From \e {a36} it follows that
	\begin{equation*}
	\mu\left(G_{l-1}\setminus \bigcup_{k=l}^n\tilde G_k\right)\ge \mu(G_{l-1})-\sum_{l\le j\le n:\, G_j\cap G_{l-1}\neq\varnothing} \mu(\tilde G_j)\ge \xi_{l-1}.
	\end{equation*}
Hence we can define $\tilde G_{l-1}=L(\xi_{l-1},G_{l-1}\setminus \bigcup_{k=l}^n\tilde G_k)$. To proceed the general case we apply the finite case that we have proved. Then for each $n$ we find a family of pairwise disjoint sets $G_k^{(n)} $, $k=1,2,\ldots, n$ such that $\mu(G_k^{(n)})=\xi_k$ for $1\le k\le n$. Applying \lem {L4} and analyzing once again the leftmost selection argument of the tilde sets, one can observe that 
\begin{equation*}
\mu(G_{k}^{(n+1)}\triangle G_k^{(n)})\le \sum_{j=k}^{n}\mu(G_{n+1}^{(n+1)}\cap G_j^{(n)})\le \xi_{n+1}.
\end{equation*}
So we conclude 
\begin{equation*}
\mu(G_{k}^{(m)}\triangle G_k^{(n)})\le \sum_{k=n+1}^m\xi_k,\quad m> n\ge k.
\end{equation*}
The last inequality implies that for a fixed $k$ the sequence  $\ZI_{G_{k}^{(m)}}$ converge in $L^1$-norm as $m\to\infty$. Moreover, one can see that the limit function is again an indicator function of a set $\tilde G_k$, and the sequence  $\tilde G_k$ satisfies the conditions of the lemma.  
\end{proof}
\begin{lemma}\label{L3}
	Let $(X,\ZM,\mu)$ be a non-atomic measure space and $f\in L^r(X)$, $1\le r<\infty$, be a boundedly supported positive function.  Then for any $\lambda>0$ there exists a measurable set $E_\lambda \subset X$ such that 
	\begin{equation}\label{a2}
	\mu(E_\lambda)\lesssim \|f\|_r^r/\lambda,\quad \{x\in X:\,M_rf(x)>\lambda\}\subset E_\lambda,
	\end{equation}
	and the function 
	\begin{equation}\label{a3}
	g(x)=f(x)\cdot \ZI_{X\setminus E_\lambda}(x)+\lambda\cdot  \ZI_{E_\lambda}(x)
	\end{equation}
	satisfies the condition
	\begin{align}
	&g(x)\le \lambda \text { a.e. on }X,\label{a44}\\
	&\langle f\rangle_{B,r}\lesssim \langle g\rangle_{B^*,r}\text{ whenever } B\in \ZB,\, B\not\subset E_\lambda.\label{a4}
	\end{align}
\end{lemma}
\begin{proof}
Applying \lem{L2} we find a sequence of pairwise disjoint maximal $\lambda$-balls $B_k$ satisfying \e {a1}. 
Thus, applying the density property (\lem {OL2}), one can conclude that
\begin{equation}\label{a43}
f(x)\le \lambda \text { for a.a. } x\in X\setminus \bigcup_kB_k^*.
\end{equation}
Given $B_k$ associate the family of balls 
\begin{equation}\label{a32}
\ZB_k=\{B\in \ZB:\, B\cap B_k^*\neq\varnothing,\, \mu(B)>2\mu(B_k^*)\}.
\end{equation}
If $\ZB_k$ is nonempty, then there is a ball $G_k\in \ZB_k$ such that
\begin{equation}\label{a38}
\mu(G_k)\le 2 \inf_{B\in \ZB_k}\mu(B).
\end{equation}
From  $\lambda$-maximality of $B_k$ and $\mu(G_k)>2\mu(B_k^*)$ we get 
\begin{equation}\label{a40}
B_k^*\subset G_k^*,\quad \langle f\rangle_{G_k^{*},r}\le \lambda.
\end{equation}
This implies 
\begin{equation}\label{a41}
\frac{1}{\lambda^r }\int_{G_k^{*}}f^r\le \mu(G_k^{*})\le c\cdot \mu(G_k),
\end{equation}
where $c>0$ is an admissible constant. Denote 
\begin{equation*}
D_1=B_1^*,\quad D_k=B_k^*\setminus \cup_{1\le j\le k-1}B_j^*,\, k\ge 2,
\end{equation*}
and consider the numerical sequence 
 $\xi_k=\frac{\delta }{\lambda^r}\int_{D_k}f^r$, $k=1,2,\ldots $, with a constant $\delta>0$. Taking into account of \e {a40}, \e {a41}, for a small admissible constant $\delta>0$ we obtain
 \begin{align*}
 \bigcup_{j:\, \mu(G_j)\le \mu(G_k),\, G_j\cap G_k\neq \varnothing}\xi_j&=\frac{\delta}{\lambda^r}\bigcup_{j:\, \mu(G_j)\le \mu(G_k),\, G_j\cap G_k\neq \varnothing}\int_{D_j}f^r\\
 &\le \frac{\delta}{\lambda^r}\int_{G_k^{*}}f^r\le c\delta\mu(G_k)\le \mu(G_k),
 \end{align*}
which gives condition \e {a36}. Since our measure space in non-atomic, applying \lem {L5}, we find pairwise disjoint subsets $\tilde G_k\subset G_k$ such that 
\begin{equation}\label{a28}
\mu(\tilde G_k)=\frac{\delta }{\lambda^r}\int_{D_k}f^r,\quad k=1,2,\ldots.
\end{equation}
Disjointness of $D_k$ implies
\begin{equation}\label{a39}
\sum_{k}\mu(\tilde G_k)= \frac{\delta }{\lambda^r}\sum_k\int_{D_k}f^r \lesssim \frac{\|f\|_{r}^r}{\lambda^r}.
\end{equation}
From the $\lambda$-maximality and disjointness properties of $B_k$, we get
\begin{equation}\label{a42}
\mu\left(\bigcup_{k} B_k^{**}\right)\lesssim \sum_k\mu\left(B_k\right)\le  \frac{1}{\lambda^r}\sum_k\int_{B_k}f^r\le \frac{\|f\|_{r}^r}{\lambda^r}.
\end{equation}
Denote 
\begin{equation*}
E_\lambda=\left(\bigcup_{k} \tilde G_k\right)\bigcup\left(\bigcup_{k} B_k^{**}\right).
\end{equation*}
From \e {a39} and \e {a42} we get $\mu(E_\lambda)\lesssim \|f\|_r^r/\lambda^r$ and \e {a43} implies \e {a44}. Hence it remains to prove that the function $g$ satisfies \e {a4}. Take a ball $B\in \ZB$ with $B\not\subset E_\lambda$. First of all observe that for each $B_k$ satisfying $B\cap B_k^*\neq \varnothing$ we have $\mu(B)> 2\mu(B_k^*)$, since otherwise we would have $B\subset B_k^{**}\subset E_\lambda$, which is not true. Thus, whenever $B\cap B_k^*\neq \varnothing$ we have $B\in \ZB_k$, then we get $\mu(G_k)\le 2 \mu(B)$ and so $\tilde G_k\subset G_k\subset B^*$. Besides, from \e {a43} and the definition of $g$ it follows that $f(x)\le g(x)$ a.e. on $X\setminus \cup_k B_k^*$. Hence, applying also \e {a28} and the disjointness of $\tilde G_k$, we obtain
\begin{align*}
\langle f\rangle_{B,r}^r&=\frac{1}{\mu(B)}\left(\int_{B\cap (\cup_k B_k^*) }f^r+\int_{B\setminus \cup_k B_k^* }f^r\right)\\
&\le \frac{1}{\mu(B)}\left(\sum_{k:\, B_k^*\cap B\neq \varnothing }\int_{B\cap D_k}f^r+\int_{B\setminus \cup_k B_k^* }g^r\right)\\
&\le \frac{1}{\mu(B)}\left(\sum_{k:\, B_k^*\cap B\neq \varnothing }\int_{ D_k}f^r+\int_{B}g^r\right)\\
&= \frac{1}{\mu(B)}\left(\sum_{k:\, B_k^*\cap B\neq \varnothing }\frac{\lambda^r\mu(\tilde G_k)}{\delta }+\int_{B}g^r\right)\\
&= \frac{1}{\delta \mu(B^*)}\left(\sum_{k:\, B_k^*\cap B\neq \varnothing }\int_{\tilde G_k}g^r+\int_{B^*}g^r\right)\\
&\lesssim  \langle g\rangle_{B^*,r}^r
\end{align*}
that implies \e {a4}, completing the proof of lemma.
\end{proof}
\section{Proof of \trm {T1}}
\begin{proof}[$L^p$-boundedness]
	For any $B\in\ZS$ we have $\langle f\rangle_{B,r}^*\le M_rf(x)$ as $x\in B$, and therefore 
\begin{equation*}
\langle f\rangle_{B,r}^*\le \langle M_rf\rangle_{B,r},\quad B\in\ZB.
\end{equation*}
Let $E_B$ be the disjoint portions of the sparse collection of balls satisfying $\mu(E_B)\ge \gamma\cdot \mu(B)$. Also suppose that $r<p<\infty$ and $q=p/(p-1)$.  Thus, for positive functions $f\in L^p$ and $g\in L^q(X)$ we have
\begin{align*}
\int_X \ZA_{\ZS,r}^*f\cdot gd\mu&\le \sum_{B\in\ZS} \langle M_rf\rangle_{B,r}\int_Bgd\mu\\
&=\sum_{B\in\ZS} \langle M_rf\rangle_{B,r}\cdot \langle g\rangle_{B,1}\cdot \mu(B)\\
&\le \gamma^{-1} \sum_{B\in\ZS} \langle M_rf\rangle_{B,r}\cdot (\mu(E_B))^{1/p}\cdot \langle g\rangle_{B,1}\cdot (\mu(E_B))^{1/q}\\
&\le \gamma^{-1}\left(\sum_{B\in\ZS} \langle M_rf\rangle_{B,r}^p\cdot\mu(E_B)\right)^{1/p}\cdot \left(\sum_{B\in\ZS} \langle g\rangle_{{B,1}}^q\cdot\mu(E_B)\right)^{1/q}\\
&\le \gamma^{-1}\|M_r(M_rf)\|_p\|M_1(g)\|_q\\
&\lesssim \|M_rf\|_{p}\cdot  \|g\|_{q}\\
&\lesssim  \|f\|_{p}\cdot  \|g\|_{q},
\end{align*}
which completes the proof of $L^p$-boundedness.
\end{proof}
\begin{proof}[Weak-$L^r$ estimate] Without loss of generality we can suppose that our measure space $(X,\ZM,\mu)$ is non-atomic, since any measure space can be extended to a non-atomic measure space by splitting the atoms as follows.  Suppose $A\subset \ZM$ is the family of atomic elements of the measure space $(X,\ZM,\mu)$, that is for any $a\in A$ we have $\mu(a)>0$ and there is no proper $\ZM$-measurable set in $a$. We can suppose that each atom is continuum and let $(a,\ZM_a,\mu_a)$ be a  a non-atomic measure space on $a\in A$ such that $\mu_a(a)=\mu(a)$. Denote by $\ZM'$ the $\sigma$-algebra on $X$ generated by $\ZM$ and by all $\ZM_a$, $a\in A$. Let $\mu'$ be the extension of $\mu$ such that $\mu'(E)=\mu_a(E)$ for any $\ZM_a$-measurable set $E\subset a$. Hence $(X,\ZM',\mu')$ gives a non-atomic extension of the measure space $(X,\ZM,\mu)$.  
	 
	Now let $f$ be a $\ZM$-measurable function. The balls are $\ZM$-measurable so they can not contain an atom $a$ partially. Thus the left and right side of inequality \e {a20} are not changed if we consider $(X,\ZM',\mu')$ instead of the initial measure space. Hence we can suppose that $(X,\ZM,\mu)$ is itself  non-atomic.	Applying \lem {L3}, we find a function $g$ satisfying the conditions of lemma. From \e {a4} we get $\langle f\rangle^*_{B,r}\le \langle g\rangle^*_{B,r}$ for any $B\in\ZS$ with $B\not\subset E_\lambda$ and therefore,
	\begin{equation*}
	\ZA_{\ZS,r}^*f(x)\le  \ZA_{\ZS,r}^*g(x),\quad x\in X\setminus E_\lambda.
	\end{equation*}
	Thus, using the $L^{2r}$ bound of $\ZA_{\ZS,r}^*$, we get
	\begin{align*}
	\mu\{x\in X:\, \ZA_{\ZS,r}^*f(x)>\lambda\}&\le \mu(E_\lambda)+\mu\{x\in X\setminus E_\lambda:\,  \ZA_{\ZS,r}^*g(x)>\lambda\}\\
	&\lesssim \frac{\|f\|_{r}^r}{\lambda^r}+\frac{1}{\lambda^{2r}}\int_{X\setminus E_\lambda}|g|^{2r}\\
	&=\frac{\|f\|_{r}}{\lambda^r}+\frac{\lambda^{r}}{\lambda^{2r}}\int_{X\setminus E_\lambda}f^r\\
	&\le \frac{2\|f\|_{r}^r}{\lambda^r}
	\end{align*}
that completes the proof of theorem. 
\end{proof}

\bibliographystyle{plain}

\end{document}